\documentclass[11pt]{amsart}
\usepackage{amsmath,amssymb,amsfonts,amsthm,amscd,indentfirst}

\numberwithin{equation}{section}
\usepackage{xcolor}

\def\<{\langle}
\def\>{\rangle}

\newtheorem{theorem}{Theorem}[section]
\newtheorem{proposition}[theorem]{Proposition}

\newtheorem{remark}[theorem]{Remark}

\newtheorem{definition}[theorem]{Definition}
\newtheorem{conjecture}[theorem]{Conjecture}

\newcommand{\rr}{\mathbb{R}}

\newcommand{\ii}{\mathfrak{i}}

\newcommand{\bpar}{\overline{\partial}}

\newcommand{\diam}{\text{diam}}

\begin{document}
\title{Remarks on Hessian quotient equations on Riemannian manifolds}
\author{Marcin Sroka}
\address{Faculty of mathematics and computer science, Jagiellonian University, \L ojasiewicza 6, 30-348, Krak\'ow, Poland}
\email{marcin.sroka@uj.edu.pl}
\thanks{The author is very grateful to Pengfei Guan for many fruitful discussions on the subject of this work. Research was supported in part by National Science Center of Poland grant no. 2021/41/B/ST1/01632.}

\begin{abstract} 
We consider Hessian quotient equations in Riemannian setting related to a problem posed by Delano\"e and Urbas. We prove unobstructed second order a priori estimate for the real Hessian quotient equation via the maximum principle argument on Riemannian manifolds in dimension two. This is achieved by introducing new test function and exploiting some fine concavity properties of quotient operator. This result demonstrates that there is intriguing difference between the real case and the complex case, as there are known obstructions for $J$-equation in complex geometry.\end{abstract}
\keywords {fully nonlinear elliptic equations; a priori estimates; Hessian quotient equations; Riemannian manifolds}

\subjclass{58J05; 35R01}
 
\maketitle
\section{Introduction}

Let $(M,g)$ be an $n$ dimensional Riemannian manifold, we consider the following equation
\begin{equation}\label{general}  F(u)=\tilde{F}(\lambda_1,...,\lambda_n)=f\end{equation} for $f \in C^\infty(M)$ where $F(u)$ factors through a symmetric function $\tilde{F}$ of eigenvalues $\lambda_i$ of endomorphism 
\begin{eqnarray*}\label{classicperturbation}  g^{-1} \circ (g + \nabla^2 u) \end{eqnarray*} for $\nabla$ being the Levi-Civita connection of $g$.

Equations of the form (\ref{general}) on Riemannian manifold, to which we refer as a global case, were studied by Delano\"e \cite{D81a, D81b, D84, D03}, Li \cite{L90}, Guan \cite{G14, G24}, Szekelyhidi \cite{Sz18}, Guo and Song \cite{GS24} and many others in a similar set up, built on the seminal work of Caffarelli, Nirenberg and Spruck  \cite{CNS85}. Their complex geometry analogues play a pivotal role for the existence of special Hermitian metrics.

\medskip

The main challenge for (\ref{general}) is the a priori estimates up to second order. It is by now settled in many situations under some \textit{structural conditions} on $F$, including $k$-Hessian operators $F=\sigma_k$ for $k=1,\cdots, n$. But it remains open for the Hessian quotient operators $F=\frac{\sigma_{k}}{\sigma_l}$ for $1\le l<k$. 
That was  signalled out as an open problem by Delano\"e in \cite{D03} and by Urbas in \cite{U02}.

In the case of  operators $\frac{\sigma_n}{\sigma_l}$ for $1 \leq l \leq n-1$ to which we refer as {\it positive Hessian quotient operators}, the zero order estimate follows trivially from the admissibility assumption on the solution and the fact that in this case the admissible cone is the positive one. As such, this estimate is free from dependence on the equation (\ref{general}). This follows from an argument which can be traced back to the work of Cheng and Yau \cite{ChY82}. The first order estimate for a wide class of operators in (\ref{general}) was proven by Urbas in \cite{U02}. The $C^0$ and $C^1$ estimates are collected in Appendix.

\bigskip

The focus of this paper is $C^2$ estimate, we experimentally deal with the operator $\frac{\sigma_2}{\sigma_1}$ in two dimensional case and establish Theorem \ref{secondorderestimate}.  The proof of Theorem \ref{secondorderestimate} relies on a new argument which we hope can be generalized to settle the problem of Delano\"e and Urbas for Hessian quotient operators in higher dimensional case. The proof consists of two main ingredients. The first one is the observation that the positive Hessian quotient operators admit a strong concavity property stated in Proposition \ref{concavityproperty}. This provides crucial term in Proposition \ref{structural}.  
The second ingredient is a new test function introduced in (\ref{newtestfunction}) involving the directional derivative of the solution associated to the eigenvector for largest eigenvalue of Hessian. This is the place at which we have to restrict to dimension two due to computational complexity. 

\medskip

A couple of  remarks should be made here. Second order estimate for equations of the form (\ref{general}) was derived by Guan \cite{G14, G24} and Szekelyhidi \cite{Sz18} under assumption on the existence of certain type of \textit{subsolution}. 
For example $k-$Hessian operators $\sigma_k$ satisfy the assumption. Szekelyhidi used this property to obtain also $C^0$ estimate for $k$-Hessian equations on Riemannian manifolds in \cite{Sz18}, hence solving $k$-Hessian equations on Riemannian manifolds. For the Hessian quotient operators, or even positive Hessian quotient operators, this is not the case if $f$ is not constant in (\ref{general}). In a similar spirit Guo and Song in \cite{GS24}, using the above mentioned conditional estimates from \cite{Sz18}, derived an equivalent condition for existence of solution to (\ref{general}). This condition being existence of yet another type of subsolution defined via inequality involving what Guo and Song call \textit{sup-slope}. Even in dimension two any of the three mentioned works does not show that the second order estimate for (\ref{Jreal}) is unobstructed and that the equations is solvable for any RHS up to the multiplicative constant.

We also note that Theorem \ref{solvability} is in striking difference with the complex case. Namely, an analogue of equation (\ref{Jreal}) when operator (\ref{realJ}) is applied to the perturbation of some K\"ahler form $\chi$ on K\"ahler manifold $(M,g)$ by complex Hessian $\ii \partial \bpar u$ has be a subject of intense study last two decades. In complex geometry this equation is known under the name $J$-equation and was introduced by Donaldson \cite{D99} and Chen \cite{Ch00}. As has been shown, this equation is not always solvable, even for constant RHS, and even on complex surfaces \cite{S20} - so in complex dimension two. Interestingly, as described above, in \cite{GS24} it was proven that the existence of smooth solutions to both real and complex positive Hessian quotient equations on manifolds is equivalent to existence of some sup-slope related subsolution. From our main result, it turns out that in dimension two for the complex equation this is an actual obstruction, while for the real equation this condition is redundant.

\medskip

The paper is organized as follows. In Section 2 we set up notations, 
and state our two main theorems. Then we proceed with proving the lower bound for the linearized operator of (\ref{realJ}) acting on largest eigenvalue of (\ref{generalperturbation}) in  Proposition \ref{structural}. We finish the proof of Theorem \ref{secondorderestimate} in Section 3 by examining the new test function (\ref{newtestfunction}). The proof of Theorem \ref{solvability} is curried over in Section 4. The $C^0$ and $C^1$ estimates are collected in Appendix. 

\section{Preliminaries}

\subsection{Positive Hessian quotient operator $\frac{\sigma_n}{\sigma_{n-1}}$}
Let $(M,g,\chi)$ be a closed, connected, $n$ dimensional Riemannian manifold $(M,g)$ endowed with a smooth, symmetric $(2,0)$ tensor $\chi$. We denote by 
\begin{eqnarray*} \nabla = \nabla^{LC}_g\end{eqnarray*} the Levi-Civita connection of $g$, i.e. the unique torsion free connection such that \begin{eqnarray*} \nabla g = 0.\end{eqnarray*} For later reference we introduce two positive constants $C(\chi,g)$ and $c(\chi,g)$ such that \begin{eqnarray} \label{gbaounds}  -c(\chi,g) g \leq \chi \leq C(\chi,g) g.\end{eqnarray}

For any function $u$ making the tensor positive definite we consider new Riemannian metric
\begin{eqnarray}\label{generalperturbation} \tilde{g}=g_u:= \chi + \nabla^2 u.\end{eqnarray}
Later, in computations, we will use the $\tilde{g}$ notation suppressing, but remembering about, dependence on $u$. 

\begin{definition}\label{def} We call a function $u\in C^2(M)$ admissible, if $g_u$ in (\ref{generalperturbation}) is positive definite.\end{definition} Together with the background $g$ the tensor $g_u$ gives rise to the endomorphism
\begin{eqnarray*} \label{endo}  g^{-1} \circ g_u : TM \longrightarrow T^*M \longrightarrow TM.\end{eqnarray*}

Consider the particular class of equations, slightly more general though then (\ref{general}) for this quotient operator, in terms of eigenvalues of $g^{-1} \circ g_u$:
\begin{eqnarray}\label{Jreal}  \begin{cases} F(u):=\frac{\sigma_n}{\sigma_{n-1}}\left(\lambda_i(g^{-1} \circ g_u)\right) = \frac{\sigma_n}{\sigma_{n-1}}(\lambda_u) = f, \\ u \text{ is admissible}\end{cases}\end{eqnarray} for a given positive $f \in C^\infty(M)$ where 
\begin{eqnarray*}\label{eigen}  \lambda:= \lambda_u = \left(\lambda_1(g^{-1} \circ g_u),...,\lambda_n(g^{-1} \circ g_u)\right)\end{eqnarray*} is decreasingly ordered vector of eigenvalues of $g^{-1} \circ g_u$ and again in calculations we use the notation $\lambda$ suppressing dependence on $u$. 

Because of the dimensional restriction in our main result - Theorem \ref{secondorderestimate} - in (\ref{Jreal}) we are interested only in this particular positive Hessian quotient operator. We denote this operator by $F$ from now on to shorten the notation. For later use we rewrite 
\begin{eqnarray} \label{realJ} \begin{gathered} F(\lambda)=\frac{\sigma_n}{\sigma_{n-1}}(\lambda) = \frac{1}{\sigma_1\left(\frac{1}{\lambda_1},...,\frac{1}{\lambda_n}\right)}:= \frac{1}{\sigma_1\left(\lambda^{-1}\right)}   \end{gathered}\end{eqnarray} where by definition 
\begin{eqnarray*} \lambda^{-1}:= \left(\frac{1}{\lambda_1},...,\frac{1}{\lambda_n}\right) \end{eqnarray*} and recall that
\begin{eqnarray*}  \sigma_k(\lambda) = \sum_{1 \leq i_1<...<i_k \leq n}\lambda_{i_1}\cdot ... \cdot \lambda_{i_k}\end{eqnarray*}
is the classic $k-$Hessian operator for $1 \leq k \leq n$. 

Our main theorem which we prove in Section 3 is the second order estimate for (\ref{realJ}).

\begin{theorem} \label{secondorderestimate}
Let $(M,g,\chi)$ be as described above and $\dim M = 2$ then there exists a constant $C$ depending only on $f$ and the background data $g$ and $\chi$ such that any solution to (\ref{Jreal}) satisfies \begin{eqnarray}\label{nabla2est}  |\nabla^2 u|_{g} \leq C.\end{eqnarray}
\end{theorem}
This together with a priori estimates of Appendix and standard ellipticity theory gives the following existence of smooth solutions to equation (\ref{Jreal}) as we demonstrate in Section 4.
\begin{theorem}\label{solvability}
Let $(M,g,\chi)$ be as described above and $\dim M = 2$, if there exists at least one admissible $v \in C^\infty(M)$ then for any $\Psi \in C^\infty(M)$ there exists a unique admissible $u \in C^\infty(M)$ solving 
\begin{eqnarray} \label{Jrealintegral}  F(u)= e^{\int_M u \: vol_g + \Psi}.\end{eqnarray} Equivalently, (\ref{Jreal}) can be solved for any positive $f \in C^\infty(M)$ up to the positive constant multiplication of $f$.
\end{theorem}

\subsection{The structural term for $\lambda_1$}

For the proof of Theorem \ref{secondorderestimate} we need the following technical estimate - Proposition \ref{structural} - which by using the extremal equation in the maximum principle argument will allow us to produce very useful term. Before stating the proposition we set up the notation, some preliminary calculations and the choice of some special coordinates which we always us whenever arguing at some point on $M$.

At the point of consideration $x_0 \in M$ we always choose normal coordinates so that:
\begin{eqnarray} \label{eigenbasic}  \begin{gathered} 
g_{ij}=\delta^i_j, \\ 
\tilde{g}_{ij} = \chi_{ij} + u_{ij}, \\ 
\tilde{g}_{ij}=\tilde{g}_{ii} \delta^i_j, \\ 
\lambda_i = \tilde{g}_{ii}, \\ 
\lambda_1 \geq \lambda_2 \geq ... \geq \lambda_n
\end{gathered} \end{eqnarray}
where everywhere \begin{eqnarray}w_{ijkl...} = \nabla_{...}\nabla_{l}\nabla_{k}\nabla_{j}\nabla_{i}w\end{eqnarray} for any function $w \in C^\infty(M)$ and in the calculations we use covariant derivatives. 

In particular, cf. \cite{U02}, we have 
\begin{eqnarray*} \label{curvature}  \begin{gathered} 
u_{ijk} = u_{kij}+C_{ijkl}u_l, \\ 
u_{ijkl}=u_{klij} + C_{ijkm}u_{ml}+C_{ijml}u_{mk}+C_{imkl}u_{mj}+C_{mjkl}u_{mi}+C_{ijklm}u_m \end{gathered}\end{eqnarray*} for $C$ being an universal symbol for any curvature dependent quantity.

In the coordinates (\ref{eigenbasic}) at $x_0 \in M$:
\begin{eqnarray}\label{eigenbasic2} \begin{gathered}  
F^{ii} = F^2 \frac{1}{\lambda_i^2}, \\ 
\frac{1}{n} = \frac{1}{n} \frac{\left(\sum_i 1 \cdot \frac{1}{\lambda_i}\right)^2}{\left(\sum_k \frac{1}{\lambda_i} \right)^2} \leq \mathcal{F}:= \sum_i F^{ii}= \frac{\sum_i \frac{1}{\lambda_i^2}}{\left(\sum_k \frac{1}{\lambda_i} \right)^2}\leq \frac{\left(\sum_i \frac{1}{\lambda_i}\right)^2}{\left(\sum_k \frac{1}{\lambda_i} \right)^2}=1, \\
F^{nn} \geq ... \geq F^{11}, \\ 
F^{nn} \geq \frac 1 n \mathcal{F} \geq F^{11}. 
\end{gathered} \end{eqnarray}

\begin{proposition}\label{structural}
For any solution $u$ of equation (\ref{Jreal}) and any $x_0 \in M$ in the coordinates (\ref{eigenbasic}) the following inequality holds 
\begin{eqnarray} \label{structuraltermestimate} L_F\left(\log(\lambda_1)\right)(x_0) \geq F^{11}(\lambda(x_0)) \frac{\tilde{g}_{11,1}^2(x_0)}{\lambda_1^2(x_0)}-C,\end{eqnarray} where $\lambda_1$ stands for the largest eigenvalue of (\ref{generalperturbation}) and is assumed to be sufficiently large, $L_F$ denotes the linearisation of $F$ at $u$ and $C$ depends on the same data as in Theorem \ref{secondorderestimate}.
\end{proposition}
\begin{proof}
In the following we compute at $x_0$ in coordinates (\ref{eigenbasic}) but we suppress the dependence of $x_0$.  Consider the quantity
\begin{eqnarray}\label{testqu}  W = \log \lambda_1.\end{eqnarray} 
Then
\begin{eqnarray}\label{secondder} \begin{gathered} W_{ii} = \frac{\lambda_{1,ii}}{\lambda_1} - \frac{\lambda_{1,i}^2}{\lambda_1^2}.\end{gathered}\end{eqnarray}

For the first term in (\ref{secondder}) we have the well known estimate, cf. \cite{BCD17}:
\begin{eqnarray}\label{eigensecder}  \begin{gathered} 
\frac{\lambda_{1,ii}}{\lambda_1} \geq \
\frac{1}{\lambda_1} \left( \tilde{g}_{11,ii} + \sum_{j>1} \frac{\tilde{g}_{j1,i}^2+\tilde{g}_{1j,i}^2}{\lambda_1-\lambda_j} \right)
= \frac{1}{\lambda_1} \left( \chi_{11,ii}+u_{11ii} + 2 \sum_{j>1} \frac{\tilde{g}_{1j,i}^2}{\lambda_1-\lambda_j} \right) \\
\geq 
\frac{1}{\lambda_1} \left( \chi_{ii,11}+u_{ii11} + 2 \sum_{j>1} \frac{\tilde{g}_{1j,i}^2}{\lambda_1-\lambda_j} \right) - \frac{C \lambda_1+C|\nabla u|_g+(\chi_{11,ii}-\chi_{ii,11})}{\lambda_1}
\\ \geq \frac{1}{\lambda_1} \left( \tilde{g}_{ii,11} + 2 \sum_{j>1} \frac{\tilde{g}_{1j,i}^2}{\lambda_1-\lambda_j} \right) - C. \end{gathered} \end{eqnarray}
In the above, last inequality follows from Proposition \ref{first}, while the previous one from (\ref{curvature}).

From differentiating equation (\ref{Jreal}) twice in direction $\partial_1$ we get consecutively:
\begin{eqnarray}\label{PDEder} \begin{gathered} 
F^{ii}\tilde{g}_{ii,1}=f_1, \\ 
F^{ii}\tilde{g}_{ii,11}=-F^{ij,rs}\tilde{g}_{ij,1}\tilde{g}_{rs,1}+f_{11}. \end{gathered}\end{eqnarray}
Another standard calculation, cf. \cite{A94}, shows that in coordinates (\ref{eigenbasic}):
\begin{eqnarray}\label{secondderexpr} \begin{gathered} F^{ij,rs}\tilde{g}_{ij,1} \tilde{g}_{rs,1}=F^{ii,jj}\tilde{g}_{ii,1}\tilde{g}_{jj,1}+\sum_{i \not = j} \frac{F^{ii}-F^{jj}}{\lambda_i-\lambda_j}\tilde{g}_{ij,1}^2 \\
=F^{ii,jj}\tilde{g}_{ii,1}\tilde{g}_{jj,1}+2 \sum_{i > j} \frac{F^{ii}-F^{jj}}{\lambda_i-\lambda_j}\tilde{g}_{ij,1}^2.\end{gathered}\end{eqnarray}

The second summand in (\ref{secondderexpr}) is known to have a sign, e.g. from (\ref{eigenbasic2}). It is of particular interest to have an explicit
lower bound. The following holds for general positive Hessian quotient operator $F=\frac{\sigma_n}{\sigma_k}$. The proof in \cite{GS} relies on arguments in \cite{GM03, GLZ09}.
\begin{proposition} \label{concavityproperty-k} For $F(\tilde g)=\frac{\sigma_n}{\sigma_k}(\tilde g)$, suppose $\tilde g>0$ which is diagonal at the point with eigenvalues $\lambda_1\ge \lambda_2\ge \cdots, \lambda_n>0$,  then there exist two positive constants $ \epsilon_0, \ \delta_0$ depending only on $n, k$ such that for any $\xi_{ij}$, $1\leq i,j \leq n$:
\begin{eqnarray*}\label{fineconcavity-k} 
- F^{\alpha\beta,\gamma \eta}\xi_{\alpha \beta}\xi_{\gamma\eta} &\ge &
 (1+\epsilon_0) F^{11} \frac{\xi_{11}^2}{\lambda_1} +\frac{1}{2} \sum_{i\ge 2} F^{ii}\frac{\xi^2_{ii}}{\lambda_i}+(1+\delta_0)\sum_{i\ge 2}F^{ii}\frac{\xi^2_{i1}}{\lambda_1}\nonumber \\
 && +\sum_{i, j\ge 2, i\neq j}F^{ii}\frac{\xi^2_{ij}}{\lambda_j}- \frac{\left(F^{ii}\xi_{ii}\right)^2}{F}.\end{eqnarray*}
\end{proposition}

As we deal only with the special operator $F=\frac{\sigma_n}{\sigma_{n-1}}$ here, the following is easy to prove.
\begin{proposition} \label{concavityproperty}
In the setting of Proposition \ref{structural} for 
\begin{eqnarray*}F(\lambda)=\frac{1}{\sigma_{1}(\lambda^{-1})}\end{eqnarray*} 
and $\xi_{i} \in \rr$ for $1 \leq i \leq n$, we have
\begin{eqnarray}\label{fineconcavity}  \begin{gathered} 
- F^{ii,jj}\xi_{i}\xi_{j} = 
 2F^{ii} \frac{\xi_{i}^2}{\lambda_i} - 2\frac{\left(F^{ii}\xi_i\right)^2}{F}.
\end{gathered}\end{eqnarray}
\end{proposition}
\begin{proof}
We denote $\kappa_i : = \frac{1}{\lambda_i}$, then one observes that 
\begin{eqnarray} F(\lambda)=\frac{1}{\sigma_1(\kappa)}.\end{eqnarray}

Computing first and second order derivatives we get:
\begin{eqnarray}\label{Fder} \begin{gathered} 
F^{ii}(\lambda) = -\frac{\frac{\sigma_{1}}{\partial \kappa_i}(\kappa)}{\sigma_1^2(\kappa)} \frac{\partial \kappa_i}{\partial \lambda_i}(\lambda)=\frac{\sigma_{0}(\kappa|i)}{\sigma_1^2(\kappa)} \cdot \frac{1}{\lambda_i^2}=F^2\frac{1}{\lambda_i^2}, \\ 
F^{ii,jj}(\lambda)=2F^3\frac{1}{\lambda_i^2\lambda_j^2}-2\delta^i_jF^2\frac{1}{\lambda_i^3}.
\end{gathered}\end{eqnarray}
Applying the above shows the following 
\begin{eqnarray}\begin{gathered} -F^{ii,jj}(\lambda)\xi_i\xi_j=-2\frac{1}{F}F^4\frac{\xi_i \xi_j}{\lambda_i^2\lambda_j^2}+2F^2\frac{1}{\lambda_i^3}\xi_i^2\\
=2F^{ii} \frac{\xi_i^2}{\lambda_i}-2\frac{(F^{ii}\xi_i)^2}{F} \end{gathered}\end{eqnarray}
thus the claimed equality (\ref{fineconcavity}) follows.
\end{proof}

Coming back to the proof of Proposition \ref{structural}, we compute LHS in (\ref{structuraltermestimate}). Applying consecutively (\ref{eigensecder}), (\ref{PDEder}), (\ref{secondderexpr}) and (\ref{fineconcavity}) we estimate it as follows:  
\begin{eqnarray}\label{main}  F^{ii}W_{ii} &=&F^{ii} \frac{\lambda_{1,ii}}{\lambda_1} - F^{ii} \frac{\lambda_{1,i}^2}{\lambda_1^2}  \nonumber \\
&\geq & \frac{1}{\lambda_1} \left( F^{ii}\tilde{g}_{ii,11} + 2 F^{ii} \sum_{j>1} \frac{\tilde{g}_{1j,i}^2}{\lambda_1-\lambda_j} \right) - F^{ii} C - F^{ii} \frac{\tilde{g}_{11,i}^2}{\lambda_1^2} \nonumber  \\
&\geq& -\frac{1}{\lambda_1} F^{ij,rs}\tilde{g}_{ij,1}\tilde{g}_{rs,1} + 2 F^{ii} \sum_{j>1} \frac{\tilde{g}_{1j,i}^2}{\lambda_1(\lambda_1-\lambda_j)}- C - F^{ii} \frac{\tilde{g}_{11,i}^2}{\lambda_1^2}  \nonumber \\
&\geq & -\frac{1}{\lambda_1} F^{ii,jj}\tilde{g}_{ii,1}\tilde{g}_{jj,1}- 2 \sum_{i > j} \frac{F^{ii}-F^{jj}}{\lambda_1(\lambda_i-\lambda_j)}\tilde{g}_{ij,1}^2 
 \\ & &+ 2 F^{ii} \sum_{j>1} \frac{\tilde{g}_{1j,i}^2}{\lambda_1(\lambda_1-\lambda_j)}- C - F^{ii} \frac{\tilde{g}_{11,i}^2}{\lambda_1^2}  \nonumber \\
&\geq& 2F^{ii} \frac{\tilde{g}_{ii,1}^2}{\lambda_i\lambda_1} - 2\frac{\left(F^{ii}\tilde{g}_{ii,1}\right)^2}{F\lambda_1} - 2 \sum_{i > j} \frac{F^{ii}-F^{jj}}{\lambda_1(\lambda_i-\lambda_j)}\tilde{g}_{ij,1}^2 
 \nonumber \\  & & + 2 F^{ii} \sum_{j>1} \frac{\tilde{g}_{1j,i}^2}{\lambda_1(\lambda_1-\lambda_j)}- C - F^{ii} \frac{\tilde{g}_{11,i}^2}{\lambda_1^2} \nonumber \\
&\geq & 2F^{ii} \frac{\tilde{g}_{ii,1}^2}{\lambda_i\lambda_1} - 2 \sum_{i > j} \frac{F^{ii}-F^{jj}}{\lambda_1(\lambda_i-\lambda_j)}\tilde{g}_{ij,1}^2 + 2 F^{ii} \sum_{j>1} \frac{\tilde{g}_{1j,i}^2}{\lambda_1(\lambda_1-\lambda_j)}- C - F^{ii} \frac{\tilde{g}_{11,i}^2}{\lambda_1^2}.  \nonumber
\end{eqnarray}

Let us focus at the moment on the term  
\begin{eqnarray} \label{fromlargeig}  I &= &
2F^{ii} \frac{\tilde{g}_{ii,1}^2}{\lambda_i\lambda_1} - 2 \sum_{i > j} \frac{F^{ii}-F^{jj}}{\lambda_1(\lambda_i-\lambda_j)}\tilde{g}_{ij,1}^2 + 2 F^{ii} \sum_{j>1} \frac{\tilde{g}_{1j,i}^2}{\lambda_1(\lambda_1-\lambda_j)} - F^{ii} \frac{\tilde{g}_{11,i}^2}{\lambda_1^2}  \nonumber \\
&=& F^{11} \frac{\tilde{g}_{11,1}^2}{\lambda_1^2}+ 2 \sum_{i>1} F^{ii} \frac{\tilde{g}_{ii,1}^2}{\lambda_i\lambda_1} 
- \sum\limits_{i \not = j; i,j \not = 1} \frac{F^{ii} -F^{jj}}{\lambda_1(\lambda_i-\lambda_j)}\tilde{g}_{ij,1}^2
- 2\sum_{j>1} \frac{F^{11}-F^{jj}}{\lambda_1(\lambda_1-\lambda_j)}\tilde{g}_{j1,1}^2  \nonumber\\
&&+2 F^{ii}\sum_{j>1} \frac{\tilde{g}_{1j,i}^2}{\lambda_1(\lambda_1-\lambda_j)} 
 -\sum_{i>1}F^{ii}\frac{\tilde{g}_{11,i}^2}{\lambda_1^2}  \nonumber \\ &= &
II + III  \end{eqnarray}
where
\begin{eqnarray} \label{positiveunused}  \begin{gathered} II = F^{11} \frac{\tilde{g}_{11,1}^2}{\lambda_1^2}+ 2 \sum_{i>1} F^{ii} \frac{\tilde{g}_{ii,1}^2}{\lambda_i\lambda_1} 
\\ - \sum\limits_{i \not = j; i,j \not = 1} \frac{F^{ii}-F^{jj}}{\lambda_1(\lambda_i-\lambda_j)}\tilde{g}_{ij,1}^2 +2 F^{ii}\sum_{j>1} \frac{\tilde{g}_{1j,i}^2}{\lambda_1(\lambda_1-\lambda_j)}\\
\geq F^{11} \frac{\tilde{g}_{11,1}^2}{\lambda_1^2} \end{gathered}\end{eqnarray}
is going to be a very positive term and, using (\ref{curvature}),
\begin{eqnarray}\label{negthirderr} \begin{gathered} III = 
2\sum_{i>1} \frac{F^{ii}-F^{11}}{\lambda_1(\lambda_1-\lambda_i)}\tilde{g}_{i1,1}^2
- \sum_{i>1}F^{ii}\frac{\tilde{g}_{11,i}^2}{\lambda_1^2} \\
= 2\sum_{i>1} \frac{F^{ii}-F^{11}}{\lambda_1(\lambda_1-\lambda_i)}\tilde{g}_{i1,1}^2
- \sum_{i>1}F^{ii}\frac{\left(\tilde{g}_{i1,1}+(\chi_{11,i}-\chi_{i1,1})+C_{11il}u_l\right)^2}{\lambda_1^2} \\
\geq 2\sum_{i>1} \frac{F^{ii}-F^{11}}{\lambda_1(\lambda_1-\lambda_i)}\tilde{g}_{i1,1}^2
- \sum_{i>1}F^{ii}\frac{\tilde{g}_{i1,1}^2}{\lambda_1^2} \\ 
 - c(\epsilon) \sum_{i>1}F^{ii}\frac{\tilde{g}_{i1,1}^2}{\lambda_1^2}
- C(\epsilon) \sum_{i>1}F^{ii}\frac{1}{\lambda_1^2} \end{gathered}\end{eqnarray}
where $c(\epsilon)$ is as small as we wish in comparison to $\epsilon>0$ at the cost of making $C(\epsilon)$ sufficiently big. 

For every $i>1$ note that
\begin{eqnarray}\label{fii-f11} \begin{gathered} \frac{F^{ii}-F^{11}}{\lambda_1(\lambda_1-\lambda_i)} = \frac{F^2}{\lambda_1(\lambda_1-\lambda_i)}(\frac{1}{\lambda_i^2} - \frac{1}{\lambda_1^2})\\ 
= \frac{F^2(\lambda_1+\lambda_i)}{\lambda_1^3\lambda_i^2}\\
=F^{ii}\frac{(\lambda_1+\lambda_i)}{\lambda_1^3} =F^{ii} \frac{1}{\lambda_1^2}(1+\frac{\lambda_i}{\lambda_1}). \end{gathered}\end{eqnarray}
Thus, after ensuring $c(\epsilon)\leq 1$ in (\ref{negthirderr}), (\ref{fii-f11}) allows one to estimate the term $III$ by
\begin{eqnarray}\label{leftover} \begin{gathered} III \geq - \frac{C}{\lambda_1^2} \mathcal{F}.\end{gathered}\end{eqnarray}

Applying (\ref{positiveunused}) and (\ref{leftover}) in (\ref{main}), we arrive at:
\begin{eqnarray} \label{structuralestimate}  \begin{gathered}  F^{ii}W_{ii} \geq F^{11} \frac{\tilde{g}_{11,1}^2}{\lambda_1^2}-C 
\end{gathered} \end{eqnarray}
as required.
\end{proof}

\section{$C^2$ estimate}
\subsection{New test quantity}
In the proof of the second order estimate Theorem \ref{secondorderestimate} we will consider a new test function 
\begin{eqnarray}\label{newtestfunction}  \tilde{Q}(x) = \log \big( \lambda_1(x)\big) + \sup\limits_{\substack{v \in T_xM, |v|_{g}=1 \\ \tilde{g}(v,v)=\lambda_1}} \phi \left( \frac{1}{2} u_v^2(x) \right).\end{eqnarray}

Suppose the function attains maximum at $x_0$ with a vector realizing maximum being $v_0 \in T_{x_0}M$. It is possible to choose coordinates as in (\ref{eigenbasic}) such that in addition 
\begin{eqnarray}\label{v=1} v_0 = \partial_{x_1}\end{eqnarray} 
at the point $x_0$. Moreover, under the assumption that $\lambda_1$ is of multiplicity one cf. (\ref{eigenvaluescomparison}) below, we extend the vector $v_0$ to 
\begin{eqnarray} \label{localV} V(x)=V^i(x) \partial_{x_i}\end{eqnarray} a smooth local $g$ unit length vector field $V$ of eigenvectors for $\lambda_1$.

Using the vector field $V$ we define locally around $x_0$ new function
\begin{eqnarray}  Q(x) = \log \lambda_1 + \phi \left(\frac{1}{2} u_V^2 \right) =\log \lambda_1 + \phi\left(\frac{1}{2} (V^ju_j)^2 \right). \end{eqnarray}
From the very definition of $\tilde{Q}$, $Q$, $V$ and assumptions on $\phi$ from below we clearly have
\begin{eqnarray}\label{Qcomparison} \tilde{Q} \geq Q\end{eqnarray} and  \begin{eqnarray} \tilde{Q}(x_0) = Q(x_0). \end{eqnarray}
Thus from (\ref{Qcomparison}), under the assumption that $x_0$ was a maximum point for $\tilde{Q}$, also $Q$ attains maximum at $x_0$.

We now assume $\phi'>0$, $\phi''=0$ and compute 
\begin{eqnarray}\label{phisecond2} \begin{gathered} F^{ii} \left( \phi \left( \frac{1}{2} (V^ju_j)^2 \right) \right)_{ii} \\
= F^{ii} \Big(\phi' \left( (V^ju_j)( V^j_iu_j + V^ju_{ji})\right) \Big)_i \\
=  F^{ii} \phi' \Big( ( V^j_iu_j + V^ju_{ji})^2 + (V^ju_j)( V^j_{ii}u_j +V^j_iu_{ji} + V^j_iu_{ji}+V^ju_{jii})\Big)
\\= 
F^{ii} \phi' \Big( ( V^j_iu_j + \delta_1^i \tilde{g}_{1i}-\chi_{1i})^2 + u_1( V^j_{ii}u_j +2V^j_iu_{ji}+u_{1ii})\Big).
\end{gathered} \end{eqnarray} 
In order to express the last quantity in (\ref{phisecond2}) purely in terms of $u$, we compute derivatives of $V$ up to second order explicitly. 

\subsection{First order derivatives of $V$}

From the very definition of $V$, as being an eigenvectorfield for $\lambda_1$, we have
\begin{eqnarray} \begin{gathered}
g^{-1} \circ \tilde{g}(V) = \lambda_1 V.
\end{gathered} \end{eqnarray}
In coordinates, satisfying (\ref{eigenbasic}) and (\ref{v=1}), we get
\begin{eqnarray} \begin{gathered}
\tilde{g}(V,\partial_k) = \lambda_1 g(V,\partial_k)
\end{gathered} \end{eqnarray}
for any fixed $1 \leq k \leq n$. 
Applying the decomposition of $V$ from (\ref{localV}) one obtains
\begin{eqnarray} \label{Veq} \begin{gathered}
V^j\tilde{g}_{jk}= \lambda_1 V^j g_{jk}.
\end{gathered} \end{eqnarray}

Taking the derivative of (\ref{Veq}) for every fixed $1 \leq i \leq n$ gives
\begin{eqnarray} \label{vfirst} \begin{gathered}
V^j_i\tilde{g}_{jk}+V^j\tilde{g}_{jk,i}= \lambda_{1,i} V^j g_{jk}+\lambda_{1} V^j_i g_{jk}.
\end{gathered} \end{eqnarray}
At the point $x_0$ equality (\ref{vfirst}) provides
\begin{eqnarray} \begin{gathered}
V^k_i\tilde{g}_{kk}+\tilde{g}_{1k,i}= \lambda_{1,i} \delta_k^1+\lambda_{1} V^k_i
\end{gathered} \end{eqnarray}
thus
\begin{eqnarray}\label{vfirstder} \begin{gathered}
V^k_i = \frac{\tilde{g}_{1k,i}}{\lambda_1-\lambda_k} 
\end{gathered} \end{eqnarray}
for any $1 \leq i \leq n$ and $k>1$. While differentiating the condition 
\begin{eqnarray}   g(V,V)=V^aV^bg_{ab}=1 \end{eqnarray} 
in $i$'th direction gives
\begin{eqnarray} \label{vlengthder} 2V^lV^l_i=0 \end{eqnarray} 
which at $x_0$ results in
\begin{eqnarray} \label{vfirstder1}  V^1_i=0\end{eqnarray}
for any $1 \leq i \leq n$.

\subsection{Second order derivatives of $V$}

Differentiating formula (\ref{vfirst}) once more,
\begin{eqnarray}\label{vsecond}  \begin{gathered}
V^j_{ii}\tilde{g}_{jk}+V^j_i\tilde{g}_{jk,i}+V^j_i\tilde{g}_{jk,i}+V^j\tilde{g}_{jk,ii}\\
= \lambda_{1,ii} V^jg_{jk}+\lambda_{1,i} V^j_i g_{jk}+\lambda_{1,i} V^j_i g_{jk}+\lambda_{1} V^j_{ii} g_{jk}
\end{gathered} \end{eqnarray} for every $1 \leq i \leq n$. At the point $x_0$ formula (\ref{vsecond}) reduces to
\begin{eqnarray}\label{vsecondder}\begin{gathered}
V^k_{ii}=\frac{1}{\lambda_1-\lambda_k} \left( 2 V^j_i \tilde{g}_{jk,i} + \tilde{g}_{1k,ii}
 - 2 \tilde{g}_{11,i} V^k_i \right)
\end{gathered} \end{eqnarray}
for $k>1$ and any $1 \leq i \leq n$. For $k=1$ and any $1 \leq i \leq n$, after differentiating formula (\ref{vlengthder}) once more in $i$'th direction, we obtain
\begin{eqnarray}\label{vsecondder1}  \begin{gathered}
V^1_{ii}=-\sum\limits_{j=2}^{n} (V^j_i)^2.
\end{gathered} \end{eqnarray}

\subsection{Second order estimate}
\begin{proof}[Proof of Theorem \ref{secondorderestimate}]

Suppose $u$ is a solution of equation (\ref{Jreal}). We consider the test quantity $\tilde{Q}$, assume it achieves maximum at $x_0 \in M$ and construct the function $Q$ as described in Section 3.1. 

We may assume 
\begin{eqnarray}\label{eigenvaluescomparison}  \lambda_1 >> \lambda_2\end{eqnarray} 
at $x_0$ since otherwise, note that from (\ref{Jreal}) $\lambda_2$ is a priori bounded from above, using in addition Proposition \ref{first} we get upper bound for $\tilde{Q}$ at $x_0$ resulting in desired estimate (\ref{nabla2est}). 

Using the extremal equation for $Q$, formulas for derivatives of $V$ plus differentiating equation (\ref{Jreal}) allows us to compute all third order derivatives of $u$, first order derivatives of $\tilde{g}$ as below. 

\medskip

Since $Q$ attain maximum at $x_0$, the following extremal equation holds for any $i\in\{1,2\}$, 
\begin{eqnarray} \label{extrsec} \begin{gathered} 0 = Q_i(x_0)=\frac{\lambda_{1,i}}{\lambda_1}+\phi'\left((V^ju_j)( V^j_iu_j + V^ju_{ji}) \right)\\
=\frac{\tilde{g}_{11,i}}{\lambda_1}+\phi'u_1(V^j_iu_j + u_{1i})=\frac{\tilde{g}_{11,i}}{\lambda_1}+\phi'u_1(V^j_iu_j + \delta_1^i \tilde{g}_{11}-\chi_{1i}). \end{gathered}\end{eqnarray}

(\ref{vfirstder}), (\ref{vfirstder1}), (\ref{vsecondder}),  (\ref{vsecondder1}) and (\ref{extrsec}) yield 
\begin{eqnarray}\label{vder2} \begin{gathered} 
V^2_1 = \frac{\tilde{g}_{12,1}}{\lambda_1-\lambda_2} 
= \frac{\tilde{g}_{11,2}+C}{\lambda_1-\lambda_2}, \\
V^2_2 = \frac{\tilde{g}_{12,2}}{\lambda_1-\lambda_2} 
=\frac{\tilde{g}_{22,1}+C}{\lambda_1-\lambda_2} , \\
V^1_1=V^1_2=0,\\
V^2_{11}=\frac{1}{\lambda_1-\lambda_2} \left( 2 V^2_1 \tilde{g}_{22,1} + \tilde{g}_{12,11}- 2 \tilde{g}_{11,1} V^2_1 \right),\\
V^2_{22}=\frac{1}{\lambda_1-\lambda_2} \left( 2 V^2_2 \tilde{g}_{22,2} + \tilde{g}_{12,22}- 2 \tilde{g}_{11,2} V^2_2 \right),\\
V^1_{11}=-(V^2_1)^2,\\
V^1_{22}=-(V^2_2)^2.
\end{gathered}\end{eqnarray}

Coupling equations from (\ref{extrsec}) for $i \in \{1,2\}$ gives the system:
\begin{eqnarray}\label{extremal2} \begin{gathered} \begin{cases}
\frac{\tilde{g}_{11,1}}{\lambda_1}+\phi'u_1(V^2_1u_2 + \tilde{g}_{11}-\chi_{11})=0\\
\frac{\tilde{g}_{11,2}}{\lambda_1}+\phi'u_1(V^2_2u_2-\chi_{12})=0 \end{cases}\end{gathered}.\end{eqnarray}
Using the formulas for first order derivatives of $V$ from (\ref{vder2}) in the system (\ref{extremal2}) gives
\begin{eqnarray}\label{extremal2c}  \begin{gathered} \begin{cases}
\tilde{g}_{11,1}=-\lambda_1\phi'u_1u_2 \frac{\tilde{g}_{11,2}}{\lambda_1-\lambda_2} -\lambda_1\phi'u_1u_2 \frac{C}{\lambda_1-\lambda_2} -\lambda_1\phi'u_1 \tilde{g}_{11}+\lambda_1\phi'u_1 \chi_{11}\\
\tilde{g}_{11,2}=-\lambda_1 \phi'u_1u_2 \frac{\tilde{g}_{22,1}}{\lambda_1-\lambda_2} -\lambda_1 \phi'u_1u_2 \frac{C}{\lambda_1-\lambda_2}+\lambda_1\phi'u_1 \chi_{12} \end{cases}\end{gathered}.\end{eqnarray}

Differentiating equation (\ref{Jreal}), 
\begin{eqnarray} \label{PDEder2}  \begin{gathered} \begin{cases}
F^{11}\tilde{g}_{11,1}+F^{22}\tilde{g}_{22,1}=f_1\\
F^{11}\tilde{g}_{11,2}+F^{22}\tilde{g}_{22,2}=f_2 \end{cases}.
\end{gathered}\end{eqnarray}
Hence
\begin{eqnarray}\label{PDEder2c} \begin{gathered} \begin{cases}
\tilde{g}_{22,1}=\frac{f_1\lambda_2^2}{f^2} - \frac{\tilde{g}_{11,1}\lambda^2_2}{\lambda_1^2} \\
\tilde{g}_{22,2}=\frac{f_2\lambda_2^2}{f^2} - \frac{\tilde{g}_{11,2}\lambda_2^2}{\lambda_1^2} \end{cases}.
\end{gathered}\end{eqnarray}

Applying first equation of (\ref{PDEder2c}) in the second equation of (\ref{extremal2c}), then the second equation of (\ref{extremal2c}) in first equation of (\ref{extremal2c}) and merging both systems (\ref{extremal2c}) and (\ref{PDEder2c}),
\begin{eqnarray}\label{extremal2cc}\begin{gathered}\quad  \quad   \begin{cases}
\tilde{g}_{11,1}(1+\phi'^2u_1^2u_2^2 \frac{\lambda_2^2}{(\lambda_1-\lambda_2)^2})=\phi'^2 u_1^2 u_2^2 \frac{\lambda_1^2}{(\lambda_1-\lambda_2)^2} (\frac{f_1\lambda_2^2}{f^2} +C) - \phi'^2u_1^2u_2 \frac{\lambda_1^2}{\lambda_1-\lambda_2} \chi_{12} \\ 
-\lambda_1\phi'u_1u_2 \frac{C}{\lambda_1-\lambda_2} -\lambda_1\phi'u_1 \tilde{g}_{11}+\lambda_1\phi'u_1 \chi_{11} \\
\tilde{g}_{11,2}=-\lambda_1 \phi'u_1u_2 \frac{1}{\lambda_1-\lambda_2}(\frac{f_1\lambda_2^2}{f^2} - \frac{\tilde{g}_{11,1}\lambda^2_2}{\lambda_1^2}) -\lambda_1 \phi'u_1u_2 \frac{C}{\lambda_1-\lambda_2}+\lambda_1\phi'u_1 \chi_{12}\\
\tilde{g}_{22,1}=\frac{f_1\lambda_2^2}{f^2} - \frac{\tilde{g}_{11,1}\lambda^2_2}{\lambda_1^2} \\
\tilde{g}_{22,2}=\frac{f_2\lambda_2^2}{f^2} - \frac{\tilde{g}_{11,2}\lambda_2^2}{\lambda_1^2} \end{cases}.\end{gathered}\end{eqnarray}

(\ref{eigenvaluescomparison}), Proposition \ref{first} and (\ref{extremal2cc}) yield the following bounds on derivatives of $\tilde{g}$: 
\begin{eqnarray}\label{gderest}  \begin{gathered}  \begin{cases}
|\tilde{g}_{11,1}|<C \phi' \lambda_1^2, \\ 
|\tilde{g}_{11,2}|<C \phi' \lambda_1, \\ 
|\tilde{g}_{22,1}|<C \phi', \\ 
|\tilde{g}_{22,2}|<C. \end{cases} \end{gathered}\end{eqnarray}
Put (\ref{gderest}) in to (\ref{vder2}),
\begin{eqnarray}\label{vder2est}  \begin{gathered}  \begin{cases}
|V^2_1|<C\phi', \\ 
|V^2_2|<C\phi'\frac{1}{\lambda_1}, \\ 
|V^1_{11}|<C\phi'^2, \\ 
|V^1_{22}|<C\phi'^2 \frac{1}{\lambda_1^2}, \\
\left|\frac{1}{\lambda_1-\lambda_2} \left( 2 V^2_1 \tilde{g}_{22,1} - 2 \tilde{g}_{11,1} V^2_1 \right)\right|<C\phi'^2 \lambda_1,\\
\left|\frac{1}{\lambda_1-\lambda_2} \left( 2 V^2_2 \tilde{g}_{22,2} - 2 \tilde{g}_{11,2} V^2_2 \right)\right|<C\phi'^2 \frac{1}{\lambda_1}. 
 \end{cases}\end{gathered}\end{eqnarray}

Applying formulas (\ref{vder2}) in (\ref{phisecond2}) we obtain:
\begin{eqnarray} \label{phisecond2expl}  \begin{gathered} 
F^{ii} \left( \phi \left( \frac{1}{2} (V^ju_j)^2 \right) \right)_{ii} 
\\ \geq
 F^{11} \phi' \left( ( V^j_1u_j + \tilde{g}_{11}-\chi_{11})^2 + u_1( V^j_{11}u_j +2V^j_1u_{j1}+u_{111})\right) \\
+ F^{22} \phi' u_1( V^j_{22}u_j +2V^j_2u_{j2}+u_{122})\\
= F^{11} \phi' ( V^2_1u_2 + \tilde{g}_{11}-\chi_{11})^2 + \Big[ \phi' u_1 \left(F^{11} u_{111} + F^{22} u_{122} \right)\Big] \\
+ \Big[ \Big[ 2 \phi' u_1(F^{11} V^2_1(-\chi_{21})+F^{22} V^2_2(\tilde{g}_{22}-\chi_{22}))\Big]\Big] 
\\ + \Big[\Big[\phi' u_1^2(F^{11} V^1_{11}+ F^{22} V^1_{22})\Big]\Big] 
+ \Big\{ \frac{\phi' u_1u_2}{\lambda_1-\lambda_2}(F^{11} \tilde{g}_{12,11}+ F^{22}\tilde{g}_{12,22})\Big\} \\
+ \Big[\Big[ \phi' u_1u_2 \left(\frac{F^{11}}{\lambda_1-\lambda_2} \left( 2 V^2_1 \tilde{g}_{22,1} - 2 \tilde{g}_{11,1} V^2_1 \right)+\frac{F^{22}}{\lambda_1-\lambda_2} \left( 2 V^2_2 \tilde{g}_{22,2} - 2 \tilde{g}_{11,2} V^2_2 \right) \right)\Big]\Big] \end{gathered} \end{eqnarray}
where in the last line we have expanded: 
\begin{eqnarray}\label{Vsecondproblematic}\begin{gathered} \phi' u_1u_2(F^{11} V^2_{11}+ F^{22} V^2_{22})=\frac{\phi' u_1u_2}{\lambda_1-\lambda_2}(F^{11} \tilde{g}_{12,11}+ F^{22}\tilde{g}_{12,22}) \\
+ \phi' u_1u_2 \left(\frac{F^{11}}{\lambda_1-\lambda_2} \left( 2 V^2_1 \tilde{g}_{22,1} - 2 \tilde{g}_{11,1} V^2_1 \right)+\frac{F^{22}}{\lambda_1-\lambda_2} \left( 2 V^2_2 \tilde{g}_{22,2} - 2 \tilde{g}_{11,2} V^2_2 \right) \right).
\end{gathered}\end{eqnarray}
Using the above estimations (\ref{vder2est}) in (\ref{phisecond2expl}), Proposition \ref{first}, taking into account (\ref{eigenbasic2}), commuting indices whenever necessary and exploiting (\ref{curvature}), applying formulas of (\ref{PDEder}) we can further estimate:
\begin{eqnarray}\label{phisecond2exp2}  \begin{gathered} 
F^{ii} \left( \phi \left( \frac{1}{2} (V^ju_j)^2 \right) \right)_{ii} 
\geq \frac{F^2}{2} \phi' + \Big[ u_1 \phi' F^{ii}\tilde{g}_{ii,1}\Big] \\
+ \Big\{ \frac{u_1u_2}{\lambda_1-\lambda_2} \phi' F^{ii} \tilde{g}_{ii,12}\Big\} - \Big[ \Big[ C \phi' |u_1|\Big]\Big]\\
\geq \frac{F^2}{2} \phi' -C \phi' |u_1| + \frac{u_1u_2}{\lambda_1-\lambda_2} \phi' (f_{12} - F^{ij,kl}\tilde{g}_{ij,1}\tilde{g}_{kl,2})\\
\geq \frac{F^2}{2} \phi' -C \phi' |u_1| - \frac{u_1u_2}{\lambda_1-\lambda_2} \phi'  F^{ij,kl}\tilde{g}_{ij,1}\tilde{g}_{kl,2}.
\end{gathered} \end{eqnarray}

Using the four estimates from (\ref{gderest}), the formula (\ref{secondderexpr}) for the quantity $F^{ij,kl}\tilde{g}_{ij,1}\tilde{g}_{kl,2}$, expressions (\ref{Fder}) and (\ref{fii-f11}) for $F^{ij,kl}$ we get:
\begin{eqnarray} \label{phisecond2expextr} \begin{gathered} - \frac{u_1u_2}{\lambda_1-\lambda_2} \phi'  F^{ij,kl}\tilde{g}_{ij,1}\tilde{g}_{kl,2} \geq -C \phi' |u_1|.
\end{gathered} \end{eqnarray}
We deduce from (\ref{phisecond2expextr}) and (\ref{phisecond2exp2}) that:
\begin{eqnarray}\label{almost} \begin{gathered} F^{ii} \left( \phi \left( \frac{1}{2} (V^ju_j)^2 \right)\right)_{ii} \geq \frac{F^2}{2} \phi' -C \phi' |u_1|.
\end{gathered} \end{eqnarray}

Finally, applying (\ref{almost}), (\ref{structuraltermestimate}), (\ref{extrsec}) for $i=1$, (\ref{vder2est}) and Proposition \ref{first} we obtain, at the maximum point of $Q$:
\begin{eqnarray}\label{qestremal} \begin{gathered} 0 \geq F^{ii}Q_{ii} 
\geq F^{11} \frac{\tilde{g}_{11,1}^2}{\lambda_1^2} + \frac{F^2}{2} \phi' - C -C \phi' |u_1| \\ 
= F^{11} \phi'^2|u_1|^2(V^2_1u_2 + \tilde{g}_{11}-\chi_{11})^2 + \frac{F^2}{2} \phi' - C -C \phi' |u_1| \\
\geq \frac{F^{2}}{2} \phi'^2|u_1|^2 + \frac{F^2}{2} \phi' - C -C \phi' |u_1| \end{gathered}\end{eqnarray}
which easily gives a contradiction with $\lambda_1$ being arbitrarily large for $\phi'=A$ a sufficiently large constant in comparison to $C$ and $\inf\limits_M f$. \end{proof}

\medskip

\begin{remark} 
The advantage of considering $u_V^2$ in (\ref{newtestfunction}) is that it narrows concentration direction to $u_1$ where $u_{11}= \lambda_1-\chi_{11}\sim \lambda_1$ as $\lambda_1$ is assumed to be large. As such, we can make effective use of Proposition \ref{structural} and critical equation (\ref{extrsec}). Restriction of $n=2$ is used to derive (\ref{gderest}) and (\ref{vder2est}). We would like to point out that vector field $V$ (more precisely $<x,V>$) was used  in \cite{ChHQ16} for interior $C^2$ estimate for Monge-Amp\'ere operator in $\rr^2$.    \end{remark}

\section{Existence of solutions to (\ref{Jreal})}
\begin{proof}[Proof of Theorem \ref{solvability}]
What we present here is a standard consequence of a priori estimates obtained above and in Appendix as well as ellipticity theory for equations of the form (\ref{Jreal}), cf. eg. \cite{D03}.

We note first of all that by using the assumption on existence of admissible $v$ for $\chi$ and considering instead of (\ref{Jrealintegral}) the new equation for $\tilde{u}$
\begin{eqnarray}\label{newJrealinteral}  F(v+\tilde{u})= e^{\int_M \tilde{u} \: vol_g + \left(\int_M v \: vol_g +\Psi\right)}\end{eqnarray}
we can assume $\chi$ is positive definite. This is the case, as if we secure statement of the theorem for (\ref{newJrealinteral}) we will have it, with $u:=v+\tilde{u}$, for (\ref{Jrealintegral}).

Thus from now on we assume $\chi$ is positive definite in (\ref{Jrealintegral}).

Under this assumption we note that for any admissible $u$ solving (\ref{Jrealintegral}) the integral 
\begin{eqnarray} \label{integral}\int_M u \: vol_g\end{eqnarray}
in under control in terms of $\Psi$. Indeed, at the maximum point of $u$ we have
\begin{eqnarray*} 0<g_u = \chi + \nabla^2 u \leq \chi\end{eqnarray*}
resulting in 
\begin{eqnarray*} e^{\int_M u \: vol_g + \Psi} \leq F(0)\end{eqnarray*}
which gives an upper bound for (\ref{integral}). Similarly, at the minimum point of $u$ we get
\begin{eqnarray*}  g_u = \chi + \nabla^2 u \geq \chi > 0\end{eqnarray*}
resulting in 
\begin{eqnarray*}  e^{\int_M u \: vol_g + \Psi} \geq F(0)\end{eqnarray*}
which gives a lower bound for (\ref{integral}).

Let us consider the continuity path of equations:
\begin{eqnarray}\label{continuity} \log \left( F(u_t)\right) = \int_M u_t \: vol_g + t \Psi\end{eqnarray}
for $t \in [0,1]$.

For $t=0$ we have a trivial, constant, solution $u_0$. 

For the openness of the set $T$ of $t$'s for which (\ref{continuity}) is solvable we consider the operator:
\begin{eqnarray*}   C^\infty(M) \supset U \ni u \longmapsto \log\left(F(u)\right)-\int_M u \: vol_g \in C^\infty(M)\end{eqnarray*}
where $U$ is the open cone of admissible $u$'s. Schauder theory for linear elliptic operators implies it is a local diffeomorphism at $u_t$ from which openness follows.  

For the closedness of $T$ we note that Theorem \ref{secondorderestimate}, Propositions \ref{first}
 and \ref{zero} as well as the above bound on $\int_M u_t \: vol_g$, in terms of $t\Psi$, paired with Evans-Krylov estimate \cite{E82, K83} provide a priori $C^{2,\alpha}$ estimate for $u_t$. This results in closedness in $C^{2,\alpha}$ space. The smoothness of the solution follows from a bootstrapping argument.
 
The uniqueness of solution to (\ref{Jrealintegral}) is a standard property. Thus the claim has been proven.
\end{proof}

\medskip

\begin{remark} The second order estimate of Theorem \ref{secondorderestimate} can also be concluded from the interior second order estimate of Heinz \cite{H59}, see also Theorem 9.4.1 in Schulz's book \cite{S90}. Heinz's interior estimate is in dimension two for Monge-Amp\'ere operator of quantity involving derivatives of $u$ up to second order. Estimate of Theorem \ref{secondorderestimate} follows from this result because, as is well known, in dimension two the operator $\frac{\sigma_2}{\sigma_{1}}$ can be rewritten as the Monge-Amp\'ere operator $\sigma_2$ but of eigenvalues of a different endomorphism. This specific interior estimate in dimension two is known to fail in higher dimensions. The merit of our argument is that the first part of it is valid in any dimension showing that the whole reasoning has potential for upgrading to general dimension. The only obstacle is getting control over derivatives of mentioned vector field in new test quantity.      \end{remark}

\begin{remark}
Equation (\ref{general}) on affine manifolds also has been a subject of previous study. In the very particular case of Monge-Amp\'ere operator already by Cheng and Yau \cite{ChY82} in '80s. In this set up one changes the connection in (\ref{classicperturbation}) for the flat connection and the metric $g$ is assumed to be the Hessian metric, i.e. locally being given by the affine Hessian of a function. When performing the maximum principle argument for second order bound for affine version of (\ref{general}) with $F$ being the positive Hessian quotient operator no qualitative difference occurs in comparison with the Riemannian set up. In particular for those equations the second order estimate was not know as well, until recently. In \cite{Lu24}, 
Lu proved interior second order estimate for operators $\frac{\sigma_n}{\sigma_{n-1}}$ and $\frac{\sigma_n}{\sigma_{n-2}}$ in $\rr^n$ for any $n>1$. Lu's interior estimate implies the second order estimate for these two positive Hessian quotient operators in the affine setting. 
\end{remark}

The following conjectures are related to the discussion of this paper.
\begin{conjecture} \label{solvconj}
For $F$ being the positive Hessian quotient operator $\frac{\sigma_n}{\sigma_l}$ for $1 \leq l \leq n-1$, on any connected, closed Riemannian manifold $(M, g)$ endowed with a symmetric $(2,0)$ tensor $\chi$ admitting at least one admissible function $v\in C^4(M)$ in the sense of Definition \ref{def}, the equation (\ref{Jreal}) is solvable for any smooth RHS up to a multiplicative constant. 
\end{conjecture}

For general Hessian quotient equation $1\le l<k$,
\begin{equation}\label{g-ke} \frac{\sigma_k}{\sigma_l}(g_u)=f,\end{equation}
we believe that there is no obstruction for $C^2$ estimate if $C^0$ estimate is in hand. 

\begin{conjecture}\label{estconj}
If $u$ is a solution of equation (\ref{g-ke})  on a closed Riemannian manifold $(M,g)$ endowed with a symmetric $(2,0)$ tensor $\chi$, then
\begin{eqnarray}\|u\|_{C^2(M)}\le C(1+\|u\|_{C^0(M)}),\end{eqnarray}
where $C$ depends only on $n, k, l, \chi, f, g$.
\end{conjecture}

\section{Appendix: $C^0$ and $C^1$ estimates revised}

\subsection{$C^0$ estimate}

The following estimate uses Cheng and Yau's argument from \cite{ChY82}.
\begin{proposition} \label{zero}
For a smooth function $u$ on $(M,g,\chi)$ such that \begin{eqnarray*}\chi+\nabla^2 u > 0\end{eqnarray*} the following bound holds: \begin{eqnarray*}osc(u) \leq C(\chi,g) \frac{(\diam_g M)^2}{2}.\end{eqnarray*}
\end{proposition}
\begin{proof}
Let \begin{eqnarray*}M=u(x_1)=\sup_M u, \: m=u(x_2)= \inf_M u\end{eqnarray*} and $\gamma$ the unit speed geodesic for $g$ on the interval $[s_1,s_2]$ between $x_1$ and $x_2$. 

Note that
\begin{eqnarray*}\begin{gathered} m-M = \int_{s_1}^{s_2} (u(\gamma(s)))_s ds = \int_{s_1}^{s_2} \int_{s_1}^{s}(u(\gamma(t)))_{tt} dt ds \\
=\int_{s_1}^{s_2} \int_{s_1}^{s} u_{ij}\left(\gamma(t)\right)\gamma_t^i(t)\gamma^j_t(t) dtds \\
\geq \int_{s_1}^{s_2} \int_{s_1}^{s} -\chi_{ij}\left(\gamma(t)\right)\gamma_t^i(t)\gamma^j_t(t) dtds \\
\geq C(\chi,g)\int_{s_1}^{s_2} \int_{s_1}^{s} -g_{ij}\left(\gamma(t)\right)\gamma_t^i(t)\gamma^j_t(t) dtds\\
=-C(\chi,g)\int_{s_1}^{s_2} \int_{s_1}^{s} |\nabla \gamma(t)|^2_g dtds 
=-C(\chi,g)\int_{s_1}^{s_2} \int_{s_1}^{s} 1 dtds \\
= -C(\chi,g)\int_{s_1}^{s_2} (s-s_1)ds = -C(\chi,g)\frac{(s_2-s_1)^2}{2} \geq -C(\chi,g)\frac{(\diam_g M)^2}{2}.\end{gathered}\end{eqnarray*} 
In the above: lower indices denote usual derivatives, the second equality follows from $x_1, \: x_2$ being extremal points of $u$, third one holds in normal coordinates for $g$ around point $\gamma(t)$ and the following one from the assumption on the Hessian of $u$ dominating $-\chi$. The claim follows from this inequality.  \end{proof}

\subsection{$C^1$ estimate}

Explicit gradient estimate is available for a wide class of equations of the form (\ref{general}), see e.g. \cite{U02}. 
For solutions in the positive cone case there is a simple proof due to Delano\"e \cite{D81a}. As we are not able to find a reference exactly for (\ref{Jreal}), we present it here.

\begin{proposition} \label{first}
For a smooth function $u$ on $(M,g,\chi)$ such that \begin{eqnarray*}\chi+\nabla^2 u > 0\end{eqnarray*} the following bound holds:\begin{eqnarray*} |\nabla u|^2_g \leq \left( C(\chi,g) \diam_g M \right)^2.\end{eqnarray*}
\end{proposition}
\begin{proof}
Take \begin{eqnarray*}W = C(\chi,g)(u-\inf_M u) + \frac{|\nabla u|^2_g}{2}= C(\chi,g)(u-\inf_M u) + \frac 1 2 g^{ij}u_iu_j.\end{eqnarray*}
At the maximum of $W$, for every $i$ separately, in coordinates (\ref{eigenbasic}) simultaneously diagonalizing $g$ and $g_u$ we get 
\begin{eqnarray*}W_i = C(\chi,g) u_{i}+ u_{ii}u_i= (C(\chi,g)+u_{ii}) \cdot (u_i) = 0\end{eqnarray*} since \begin{eqnarray*}C(\chi,g)+u_{ii} \geq \chi_{ii} + u_{ii} > 0\end{eqnarray*} this implies \begin{eqnarray*}\nabla u = 0.\end{eqnarray*} It thus follows that \begin{eqnarray*}\sup\limits_M W = \sup\limits_M C(\chi,g)(u-\inf_M u) \leq \left( C(\chi,g) \right)^2 \frac{\left(\diam_g M\right)^2}{2}.\end{eqnarray*}

This in turn means \begin{eqnarray*} \sup\limits_M \frac{|\nabla u|^2_g}{2} = \sup\limits_M (W -u+\inf_M u) \leq \sup\limits_M W \leq \left( C(\chi,g) \right)^2  \frac{(\diam_g M)^2}{2} \end{eqnarray*} which gives the claim. 
\end{proof}

\end{document}